\documentclass[11pt]{article}

\usepackage[a4paper,margin=1in]{geometry}
\usepackage{amsmath,amssymb,amsthm,mathtools}
\usepackage{microtype}
\usepackage{enumitem}
\usepackage{comment}
\usepackage{xcolor}

\newtheorem{theorem}{Theorem}
\newtheorem{lemma}[theorem]{Lemma}
\newtheorem{claim}[theorem]{Claim}
\newtheorem{proposition}[theorem]{Proposition}
\newtheorem{corollary}[theorem]{Corollary}
\newtheorem{problem}[theorem]{Problem}
\newtheorem{conjecture}[theorem]{Conjecture}
\newtheorem{remark}[theorem]{Remark}

\definecolor{myblue}{RGB}{0,50,250}

\definecolor{mypink}{RGB}{220,80,150}
\usepackage[
    colorlinks=true,
    linkcolor=myblue,   
    citecolor=mypink,   
    urlcolor=red    
]{hyperref}

\newcommand{\avgd}{\overline d}

\newcommand{\E}{\mathbb E}

\title{A Phase Transition for Small Dense Subhypergraphs}
\author{Peiru Kuang\footnote{School of Mathematical Sciences, Shanghai Jiao Tong University, Shanghai 200240, China. Supported by Shanghai Institute for Mathematics and Interdisciplinary Sciences, SIMIS-ID-26-AMS-003. Email: peiru\_k@sjtu.edu.cn}
\and
Yan Wang\footnote{School of Mathematical Sciences, Shanghai Jiao Tong University, Shanghai 200240, China. Supported by National Key R\&D Program of China under Grant No. 2022YFA1006400 and National Natural Science Foundation of China under Grant No. 12571376. Email: yan.w@sjtu.edu.cn (corresponding author).}}
\date{}

\begin{document}
\maketitle

\begin{abstract}
The local--global principle, which concerns the relationship between local structure and global parameters, has attracted considerable attention in extremal combinatorics over the past few decades. In this paper, we study how global density forces small dense subhypergraphs in uniform hypergraphs. For fixed $r\ge3$ and $s>1$, let $t_r(n,d,s)$ be the smallest integer $t$ such that every $n$-vertex $r$-graph of average degree at least $d$ contains a nonempty subhypergraph on at most $t$ vertices with average degree at least $s$. We show that the behavior of $t_r(n,d,s)$ undergoes a phase transition at $s=r/(r-1)$. We determine $t_r(n,d,s)$ and obtain asymptotically sharp bounds in several parameter regimes. This answers, up to polylogarithmic factors, a question of Feige and Wagner that was later restated as Problem~3.3 by Janzer, Sudakov and Tomon. In particular, when $r=3$ and $s=2$, our result implies a conjecture of Feige.
\end{abstract}

\section{Introduction}
An \emph{$r$-uniform hypergraph}, or $r$-graph, is a pair $H=(V(H),E(H))$, where $V(H)$ is a finite set and $E(H)\subseteq\binom{V(H)}{r}$. An $r$-graph $F$ is a \emph{subhypergraph} of an $r$-graph $H$ if $V(F)\subseteq V(H)$ and $E(F)\subseteq E(H)$. For a nonempty $r$-graph $H$, write $\avgd(H)=re(H)/|V(H)|$. 
In this paper, we study the following problem for uniform hypergraphs. If an $n$-vertex $r$-uniform hypergraph has average degree at least $d$, how small a subhypergraph of average degree at least $s$ must it contain? More precisely, for an integer $r\ge2$ and real numbers $d\ge s>1$, let $t_r(n,d,s)$ be the smallest integer $t$ such that every $n$-vertex $r$-graph $H$ with $\avgd(H)\ge d$ contains a nonempty subhypergraph $H'$ with $|V(H')|\le t$ and $\avgd(H')\ge s$, where $d$ is allowed to depend on $n$. 

Many problems in extremal graph theory fall within the framework of the local--global principle, which concerns the relationship between local structure and global parameters; see, for example, \cite{BucicSudakov,JST,LinialLocalGlobal,LinialRabinovich}. Here the direction is global-to-local: a global average-degree condition is used to force a dense subhypergraph on a small number of vertices. The graph case already exhibits two different phenomena. When $s=2$, the problem is equivalent to the classical girth problem. Indeed, $t(n,d,2):=t_2(n,d,2)$ is precisely the maximum possible girth of an $n$-vertex graph of average degree at least $d$. When $d=2$, we have $t(n,2,2)=n$. For $d>2$, the Moore bound, together with classical constructions of graphs of large girth~\cite{LPS,Margulis,Morgenstern}, yields $t(n,d,2)=\Theta(\log_{d-1}n)$, although determining the optimal leading constant remains a major open problem. Alon, Hoory and Linial~\cite{AlonHooryLinial} extended the Moore bound to irregular graphs; see also Ajesh Babu and Radhakrishnan~\cite{BabuRadhakrishnan} for an entropy-based proof. 
This problem is also closely related to the densest $k$-subgraph problem, which asks for a $k$-vertex subgraph of maximum average degree; see~\cite{BhaskaraEtAl,Khot}.
Related problems replace average degree by other degree restrictions. Erd\H{o}s, Faudree, Rousseau and Schelp~\cite{EFRS,EFRS2} studied the analogous problem with the stronger requirement $\delta(H')\ge s$, where $s\ge2$ is an integer. Sauermann~\cite{Sauermann} proved their conjecture that every sufficiently large $n$-vertex graph with at least $(s-1)n-\binom{s}{2}+2$ edges contains a subgraph of minimum degree at least $s$ on at most $(1-\varepsilon_s)n$ vertices, for some $\varepsilon_s>0$. A further strengthening asks for a small $s$-regular subgraph. Without a restriction on its order, the corresponding existence problem for fixed $s\ge3$ is the Erd\H{o}s--Sauer problem \cite{Erdos1975}, resolved by Janzer and Sudakov~\cite{JanzerSudakovRegular}.


For $s>2$, the dependence on $n$ and $d$ becomes polynomial rather than logarithmic. A random graph, followed by the deletion of all small dense configurations, gives $t(n,d,s)\ge c_snd^{-s/(s-2)}$. Feige and Wagner~\cite{FW} conjectured that this lower bound is sharp up to polylogarithmic factors. Janzer, Sudakov and Tomon~\cite{JST} proved their conjecture by showing that $t(n,d,s)\le nd^{-s/(s-2)}(\log d)^{O_s(1)}$ whenever $d\le n^{(s-2)/s}$. They further proved that, for every real $s>2$ and every $\varepsilon>0$, there exists a constant $T=T(s,\varepsilon)$ such that every sufficiently large $n$-vertex graph of average degree at least $n^{1-2/s+\varepsilon}$ contains a subgraph of average degree at least $s$ on at most $T$ vertices, which confirmed a conjecture of Verstra\"ete. Jiang and Newman~\cite{JiangNewman} had previously proved the same result when $s$ is an integer. 

Feige and Wagner~\cite{FW} proposed the following hypergraph problem, which was later restated by Janzer, Sudakov and Tomon~\cite{JST}.
\begin{problem}[\cite{FW}, \cite{JST}]\label{prob:JST}
For fixed $r\ge3$ and $s>1$, determine the asymptotic behavior of $t_r(n,d,s)$.
\end{problem}
In this paper, we resolve Problem~\ref{prob:JST} up to polylogarithmic factors and obtain asymptotically sharp bounds in several parameter ranges.
The value $c_r=r/(r-1)$ is a natural structural threshold for the problem. Indeed, writing \emph{excess} $\xi_r(H):=(r-1)e(H)-|V(H)|$, we have $\avgd(H)<c_r$, $\avgd(H)=c_r$, and $\avgd(H)>c_r$ correspond, respectively, to negative, zero, and positive excess. For connected $r$-graphs, these are the hypertree, unicyclic, and complex regimes. Our results show that the behavior of $t_r(n,d,s)$ changes precisely across the three ranges $s<c_r$, $s=c_r$, and $s>c_r$. This trichotomy is reflected in a phase transition in the behavior of $t_r(n,d,s)$ at $s=c_r$: our results exhibit qualitatively different asymptotics in the subcritical, critical, and supercritical regimes. 

\subsection{The case \texorpdfstring{$s>c_r$}{s > c(r)}}

Define $\alpha_{r,s}=s/((r-1)s-r)$.

\begin{theorem}\label{thm:main}
For every fixed integer $r\ge3$ and real number $s>c_r$, there is a constant $C=C(r,s)$ such that the following holds for every sufficiently large $d$. Let $H$ be an $n$-vertex $r$-graph of average degree at least $d$, where $d\le n^{1/\alpha_{r,s}}$. Then $H$ contains a nonempty subhypergraph $H'$ of average degree at least $s$ such that
$
 |V(H')|\le nd^{-\alpha_{r,s}}(\log d)^C.
$
\end{theorem}
Let $d_0=d_0(r,s)$ be such that Theorem~\ref{thm:main} holds for every $d\ge d_0$. The restriction that $d$ be sufficiently large is not crucial. Indeed, if $s\le d<d_0$, then the hypergraph itself has average degree at least $s$, and hence $t_r(n,d,s)\le n\le d_0^{\alpha_{r,s}}nd^{-\alpha_{r,s}}$. 

Note that Theorem~\ref{thm:main} is stated under the assumption $d\le n^{1/\alpha_{r,s}}$, it also applies to $r$-graphs of average degree greater than $n^{1/\alpha_{r,s}}$ by taking
$d=n^{1/\alpha_{r,s}}$. Thus, for all sufficiently large $n$, every such $r$-graph contains a subhypergraph of average degree at least $s$ on at most $(\log n)^{O_{r,s}(1)}$ vertices.


A standard random construction gives $t_r(n,d,s)\ge c n d^{-\alpha_{r,s}}$ for some $c=c(r,s)>0$. Thus the dependence on $n$ and $d$ in Theorem~\ref{thm:main} is optimal up to polylogarithmic factors. However, at the threshold $d=n^{1/\alpha_{r,s}}$, the term $nd^{-\alpha_{r,s}}$ equals one and does not capture the full behavior. The following stronger lower bound exhibits the additional logarithmic term that arises near this threshold.
\begin{theorem}\label{thm:lower}
For every fixed integer $r\ge3$ and real number $s>c_r$, there is a constant $c=c(r,s)>0$ such that, for all sufficiently large $n$ and every $s\le d\le n^{1/\alpha_{r,s}}$, writing $w=nd^{-\alpha_{r,s}}$ and $\Lambda=\log(2+nd)$, we have
$$
t_r(n,d,s)\ge c\left(
w+\frac{\Lambda}{\log(2+\Lambda/w)}
\right).
$$
\end{theorem}

If $w\ge\Omega(\Lambda)$, then $\log(2+\Lambda/w)=\Theta(1)$, so Theorem~\ref{thm:lower} gives $t_r(n,d,s)=\Omega_{r,s}(w)$. Together with Theorem~\ref{thm:main}, this determines $t_r(n,d,s)$ up to a polylogarithmic factor in this range. If $w=o(\Lambda)$, then $\Lambda/\log(2+\Lambda/w)$ exceeds $w$ by an unbounded factor. 
In fact, throughout the range $(n/\log n)^{1/\alpha_{r,s}}\ll d\le n^{1/\alpha_{r,s}}$, we have $w=o(\Lambda)$, so the second term in Theorem~\ref{thm:lower} exceeds $w$ by an unbounded factor. At the endpoint $d=n^{1/\alpha_{r,s}}$, we have $w=1$ and $\Lambda=\Theta_{r,s}(\log n)$, and hence $t_r(n,d,s)=\Omega_{r,s}(\log n/\log\log n)$. Thus some logarithmic factor in Theorem~\ref{thm:main} is unavoidable.
Janzer, Sudakov and Tomon~\cite{JST} also observed that a logarithmic correction is necessary in the graph case when $d=\Omega(n^{(s-2)/s})$.

Theorem~\ref{thm:lower} shows that no bound independent of $n$ is possible at $d=n^{1/\alpha_{r,s}}$. Such a bound does hold if the exponent of $n$ is increased by any fixed $\varepsilon>0$.

\begin{theorem}\label{thm:high-density}
For every fixed integer $r\ge3$, real number $s>c_r$ and $\varepsilon>0$, there is a constant $T=T(r,s,\varepsilon)$ such that the following holds for all sufficiently large $n$. Every $n$-vertex $r$-graph of average degree at least $n^{1/\alpha_{r,s}+\varepsilon}$ contains a nonempty subhypergraph of average degree at least $s$ on at most $T$ vertices. Moreover, for fixed $r$, $s$ and $\varepsilon$, such a subhypergraph can be found in time $n^{O_{r,s,\varepsilon}(1)}$.
\end{theorem}
There is also an exact Tur\'an reformulation. For fixed $r$ and $s$, and a positive integer $t$, let $\mathcal F_{r,s,t}$ be the family of all $r$-graphs on at most $t$ vertices with average degree at least $s$ and let $\operatorname{ex}_r(n,\mathcal F)$ denote the maximum
number of edges in an $n$-vertex $r$-graph containing no member of $\mathcal F$ as a subhypergraph. Then $t_r(n,d,s)>t$ if and only if $\operatorname{ex}_r(n,\mathcal F_{r,s,t})\ge nd/r$. Thus the problem can be viewed as a local-density version of the Brown--Erd\H{o}s--S\'os problem \cite{BrownErdosSos}. Related extremal problems with prescribed numbers of vertices and edges were studied by Alon and Shapira~\cite{AlonShapira}, and the corresponding graph problem was studied by Jiang and Newman~\cite{JiangNewman}. Since $1+1/\alpha_{r,s}=r-r/s$, Theorem~\ref{thm:lower} implies that, for every fixed positive integer $t$ and all sufficiently large $n$,
$$
\operatorname{ex}_r(n,\mathcal F_{r,s,t})\ge \frac{1}{r}n^{r-r/s}.
$$

Conversely, Theorem~\ref{thm:high-density} implies that, for every $\varepsilon>0$, there is a positive integer $T=T(r,s,\varepsilon)$ such that, for all sufficiently large $n$,
$$
\operatorname{ex}_r(n,\mathcal F_{r,s,T})<\frac{1}{r}n^{r-r/s+\varepsilon}.
$$

Hence
$$
\lim_{t\to\infty}\limsup_{n\to\infty}\frac{\log\operatorname{ex}_r(n,\mathcal F_{r,s,t})}{\log n}=r-\frac{r}{s}.
$$

Following Feige~\cite{FeigeEven}, the notation $\widetilde O(f)$ suppresses a multiplicative $O(\log n)$ factor. First, for $r\ge3$, the case $s=2$ is closely related to the even cover problem. An \emph{even cover} is a nonempty set of hyperedges in which every vertex has even degree. Early work on short linear dependencies and even covers includes that of Naor and Verstra\"ete~\cite{NaorVerstraete}. These questions are also closely connected to the refutation of random and semirandom constraint satisfaction problems and to locally decodable codes; see, for example, Feige, Kim and Ofek~\cite{FeigeKimOfek} and Alrabiah, Guruswami, Kothari and Manohar~\cite{AGKM}. Feige~\cite{FeigeEven} conjectured that every $n$-vertex $r$-graph of average degree $d$ contains an even cover with $\widetilde O(nd^{-2/(r-2)})$ edges. This conjecture was proved by Guruswami, Kothari and Manohar~\cite{GKM}; Hsieh, Kothari, Mohanty, Munh\'a Correia and Sudakov~\cite{HKMCS} later gave a purely combinatorial proof and improved the logarithmic factor for odd uniformity. More recently, Bandeira, Kunisky,
Nizi\'c-Nikolac, Pesenti and Wang~\cite{BKKNPW}, in the current version of their paper, and Schmidhuber and Hastings~\cite{SchmidhuberHastings} proved the sharp hypergraph
Moore bound for every $r\ge3$.

Feige~\cite[Conjecture~1.7]{FeigeEven} also proposed the following conjecture on small dense subhypergraphs of $3$-uniform hypergraphs.

\begin{conjecture}[Feige~\cite{FeigeEven}]\label{conj}
For all sufficiently large $n$ and every $2\le d\le n^{1/2}$,
$$
t_3(n,d,2)=\widetilde O(n/d^2).
$$
\end{conjecture}

As a consequence of Theorem~\ref{thm:main}, we have the following, which resolves Conjecture~\ref{conj}. 
\begin{corollary}
For all sufficiently large $n$ and every $2\le d\le n^{1/2}$,
$$
t_3(n,d,2)\le nd^{-2}(\log d)^{O(1)}.
$$
\end{corollary}
Moreover, Theorem~\ref{thm:high-density} gives $t_3(n,d,2)=O_\varepsilon(1)$ whenever $d\ge n^{1/2+\varepsilon}$.

\subsection{The case \texorpdfstring{$s\le c_r$}{s <= c(r)}}

The remaining two ranges admit exact answers. 
Set $M=\lceil s/(r-(r-1)s)\rceil$.

\begin{theorem}\label{thm:subcritical}
Fix an integer $r\ge3$ and a real number $1<s<c_r$. Then $t_r(n,d,s)\le (r-1)M+1$ for every $d\ge s$. Moreover, $t_r(n,s,s)=(r-1)M+1$ for all sufficiently large $n$.
\end{theorem}

For $x\ge0$, define
$$
 \Phi_n(x):=
 \begin{cases}
 n,&x=0,\\[1mm]
 \displaystyle\frac{\log(1+nx)}{\log(1+x)},&x>0.
 \end{cases}
$$

Thus $\Phi_n(x)=\Theta(n)$ when $nx=O(1)$, $\Phi_n(x)=\Theta(\log(nx)/x)$ when $n^{-1}\ll x\ll1$, and $\Phi_n(x)=\Theta(1+\log n/\log(x+1))$ when $x\ge1$.

\begin{theorem}\label{thm:critical}
Fix an integer $r\ge3$. For every $n$ and every $d\ge c_r$, let $\gamma=(r-1)d-r$. We have $t_r(n,d,c_r)=\Theta_r(\Phi_n(\gamma))$. Moreover, $t_r(n,c_r,c_r)=n$ whenever $(r-1)\mid n$.
\end{theorem}

Thus the answer is linear at $d=c_r$, but it becomes logarithmic once $d-c_r$ is bounded away from zero, and eventually becomes bounded as the excess density grows.

\subsection{Proof Overview}

The proof splits into three ranges, which naturally arise from the \emph{excess} $\xi_r(H)=(r-1)e(H)-|V(H)|$. For a connected $r$-graph, excess $-1$, zero excess, and positive excess correspond to the hypertree, unicyclic, and complex cases, respectively; this trichotomy is standard in the study of random hypergraphs~\cite{KaronskiLuczak,SchmidtPruzanShamir}. Since the sign of $\xi_r(H)$ is determined by whether $\avgd(H)$ is below, equal to, or above $c_r=r/(r-1)$, this structural trichotomy underlies the phase transition in the behavior of $t_r(n,d,s)$ at $s=c_r$. Consequently, we consider the three ranges $s<c_r$, $s=c_r$, and $s>c_r$ separately.

The subcritical case is handled by a simple edge-ordering argument, with a linear path construction giving the matching lower bound in Theorem~\ref{thm:subcritical}.
At the critical value $s=c_r$, the problem becomes a girth
problem in the incidence graph (see Lemma~\ref{lem:incidence-cycle}): An $r$-graph contains a subhypergraph of average degree at least $c_r$ if and only if its incidence graph contains a cycle. For the upper bound (see Proposition~\ref{prop:critical-upper}), we apply the bipartite Moore bound of Hoory~\cite{Hoory}; see also~\cite{BabuRadhakrishnan}. For the lower bound (see Proposition~\ref{prop:critical-lower}), near the threshold we take high-girth cubic graphs supplied by Linial and Simkin~\cite{LinialSimkin} and subdivide their edges, while farther from the threshold we use the random deletion argument in Lemma~\ref{lem:random-berge-girth}.

We now turn to the supercritical range $s>c_r$. The lower bound is obtained by an alteration argument. 
For the upper bound, we prove Proposition~\ref{prop:stable-supercritical} by induction on $r$. The main difficulty is that the reduction in uniformity changes the target average degree and introduces additional vertices when the resulting edges are lifted back to the original hypergraph. Consequently, an induction statement for a single fixed target average degree is not sufficient. The base case when $r=2$ is given by Lemma~\ref{lem:JST-uniform}, which is a uniform version of a result of Janzer, Sudakov and Tomon~\cite{JST}. For $r\ge3$, the induction splits into three cases according to the value of $u_0$ compared to $r/(r-2)$.

When $u_0<r/(r-2)$, we can directly reduce to the graph case. If some pair of vertices has sufficiently large codegree, a bounded number of hyperedges containing that pair already form a subhypergraph of the required average degree. Otherwise, every pair has bounded codegree. We then choose a pair $p(e)\subseteq e$ for each hyperedge $e$ and form the graph whose edges are the distinct pairs $p(e)$. This graph has average degree $\Omega_{r,u_0}(d)$. Applying Lemma~\ref{lem:JST-uniform} with target $q=2u/(r-(r-2)u)$, we obtain a small subgraph of average degree at least $q$. Extending a suitable collection of its edges to hyperedges of the original $r$-graph gives the required subhypergraph. 

For $u_0\ge r/(r-2)$, the main ingredient is Lemma~\ref{lem:compression}. Either it gives the required subhypergraph directly or there exist a small set $X$ and an auxiliary $(r-1)$-graph $J$ on
$V(G)\setminus X$ with average degree $\Omega(pd)$, such that every edge $P\in E(J)$ extends to a distinct edge $P\cup\{x_P\}\in E(G)$ for some $x_P\in X$.
The value $r/(r-2)$ appears because the natural target after reducing the uniformity from $r$ to $r-1$ is approximately $(r-1)u/r$. When $u_0=r/(r-2)$, this target is exactly $c_{r-1}$, so the reduction reaches the critical regime for $(r-1)$-graphs and the supercritical induction hypothesis is no longer available. We therefore treat this boundary case separately using Lemma~\ref{lem:excess-localization}, whose excess estimate is strong enough to account for the additional vertices introduced when the edges of the $(r-1)$-graph are extended back to the original $r$-graph.

When $u_0>r/(r-2)$, we have $(r-1)u_0/r>c_{r-1}$, so the induction hypothesis becomes available. Another difficulty is caused by the vertices of $X$. If the induction hypothesis were applied with the natural target $q^*=(r-1)u/r$, then a subhypergraph on a vertex set $W$ would only guarantee
$
re_J(W)\ge u|W|,
$
whereas Claim~\ref{clm:injective-lifting} requires
$
re_J(W)\ge u(|W|+|X|).
$
We therefore apply the induction hypothesis with the slightly larger target
\[
q=\frac{(r-1)u}{r}+(\log d)^{-A/2}.
\]
This increase allows us to choose $B$ such that
$rqB/((r-1)(B+1))\ge u$. We repeatedly apply the induction
hypothesis until either the lifting condition already holds or the selected vertex set $W$ satisfies $|W|\ge B|X|$, in which case the same inequality implies the lifting condition.

Finally, the target $q$ depends on $d$ and converges to $(r-1)u_0/r$. Hence an induction statement proved separately for each fixed target would not provide constants uniform in $d$. This is the reason why we need a stable version of the supercritical regime (see Proposition~\ref{prop:stable-supercritical}). Taking $u_0=u=s$ then proves Theorem~\ref{thm:main}.

\section{Preliminaries}
We first recall the notation and terminology used throughout the paper. For a positive integer $m$, write $[m]=\{1,\ldots,m\}$. For an $r$-graph $H$ and a set $S\subseteq V(H)$, let $d_H(S)=|\{e\in E(H):S\subseteq e\}|$ denote the \emph{codegree} of $S$. The \emph{girth} of $H$ is denoted by $g(H)$. 
The \emph{incidence graph} $B_H$ of an $r$-graph $H$ is the bipartite graph with vertex classes $V(H)$ and $E(H)$, where $v\in V(H)$ is adjacent to $e\in E(H)$ if and only if $v\in e$. We say that $H$ is \emph{connected} if $B_H$ is connected. The \emph{girth} $g(G)$ of a graph $G$ is the length of its shortest cycle, and we define $g(G)=\infty$ if $G$ is acyclic.
For an $r$-graph $F$, define its \textit{excess} by $\xi_r(F)=(r-1)e(F)-|V(F)|$. The notion was introduced by Wright~\cite{Wright}. 

For the upper bound in the critical case, we use the following bipartite Moore bound of Hoory~\cite{Hoory}; also see Babu and Radhakrishnan~\cite{BabuRadhakrishnan}.
\begin{lemma}[\cite{BabuRadhakrishnan,Hoory}]\label{lem:bipartite-moore}
Let $B$ be a bipartite graph with parts $X$ and $Y$, where $|X|=N$ and $|Y|=M$, and suppose that $B$ has minimum degree at least two. Let $a=e(B)/N$ and $b=e(B)/M$. If $B$ has girth $2\ell$, then
$$
N\ge
\sum_{i=0}^{\ell-1}
(b-1)^{\lceil i/2\rceil}
(a-1)^{\lfloor i/2\rfloor}
$$
and
$$
M\ge
\sum_{i=0}^{\ell-1}
(a-1)^{\lceil i/2\rceil}
(b-1)^{\lfloor i/2\rfloor}.
$$
\end{lemma}

For the lower-bound construction near the critical threshold, we use the existence of cubic graphs with logarithmic girth due to Linial and Simkin~\cite{LinialSimkin}.
\begin{lemma}[Linial and Simkin~\cite{LinialSimkin}]\label{lem:cubic-girth}
There is an absolute constant $c>0$ such that, for every sufficiently large even integer $u$, there is a cubic graph on $u$ vertices with girth at least $c\log u$.
\end{lemma}

The base case of our supercritical argument is supplied by the following two results of Janzer, Sudakov and Tomon~\cite{JST} for graphs.

\begin{theorem}[Janzer, Sudakov and Tomon \cite{JST}]\label{thm:JST}
For every $s>2$, there is a constant $C(s)$ such that the following holds for all $d\ge s$. Let $G$ be an $n$-vertex graph with average degree at least $d$, where $d\le n^{(s-2)/s}$. Then there is a nonempty set $R\subseteq V(G)$ of size at most $nd^{-s/(s-2)}(\log d)^{C(s)}$ such that $G[R]$ has average degree at least $s$.
\end{theorem}

\begin{theorem}[Janzer, Sudakov and Tomon~\cite{JST}]\label{thm:JST-high}
For every $s>2$ and $\varepsilon >0$, there is a positive integer $t$ such that the following holds for all sufficiently large $n$. Let $G$ be an $n$-vertex graph of average degree $d \ge n^{1-2/s +\varepsilon}$. Then there is a nonempty set $R\subseteq V(G)$ of size at most $t$ such
that $G[R]$ has average degree at least $s$.
\end{theorem}

\section{Subcritical phase: \texorpdfstring{$1<s<c_r$}{1 < s < cr}}
First we consider the range $1<s<c_r$ and prove Theorem \ref{thm:subcritical}. Let $\Delta=r-(r-1)s>0$ and $f(k)=\frac{rk}{(r-1)k+1}$ for $k\ge1$. Note that $f$ is strictly increasing on $k$, and $f(k)\ge s$ if and only if $k\Delta\ge s$. Hence $M=\lceil s/\Delta\rceil$ is the smallest positive integer such that 
\begin{equation}\label{f(M)}
f(M-1)<s\le f(M).    
\end{equation}

Since $t_r(n,d,s)\le t_r(n,s,s)$ for every $d\ge s$, it suffices to prove the upper bound when $d=s$.
\begin{lemma}
$t_r(n,s,s)\le (r-1)M+1$ for every $n$.
\end{lemma}
\begin{proof}
Let $H$ be an $n$-vertex $r$-graph with
$\avgd(H)\ge s$. Choose a connected component $C$ of $H$ of maximum average degree. Then $\avgd(C)\ge\avgd(H)\ge s$.
Write $m=e(C)$. Since $C$ is connected, its edges can be ordered as $e_1,\ldots,e_m$ so that $e_i\cap(e_1\cup\cdots\cup e_{i-1})\ne\varnothing$
for every $i\ge2$.

If $m<M$, then we can take $F=C$, since $|V(C)|\le (r-1)m+1\le(r-1)M+1$. If $m\ge M$, the first $M$ edges $e_1,\ldots,e_M$ span at most $(r-1)M+1$ vertices and have average degree at least $f(M)\ge s$ by \eqref{f(M)}. Thus
$t_r(n,s,s)\le (r-1)M+1$.
\end{proof}

We next prove the matching lower bound when $d=s$. 
\begin{proposition}
\label{prop:subcritical-lower}
For all sufficiently large $n$,
$
t_r(n,s,s)\ge (r-1)M+1.
$
\end{proposition}
\begin{proof}
Let $m=\lceil sn/r\rceil$. Since $m\le sn/r+1$, we have $(r-1)m+1\le n-\frac{\Delta n}{r}+r\le n$ for all sufficiently large $n$. Choose distinct vertices $v_0,\ldots,v_m$ and pairwise disjoint sets $A_1,\ldots,A_m$, each of size $r-2$ and disjoint from $\{v_0,\ldots,v_m\}$. For $i\in[m]$, let $e_i=\{v_{i-1},v_i\}\cup A_i$. Add isolated vertices so that the resulting $r$-graph $H$ has $n$ vertices.
Then $\avgd(H)=rm/n\ge s$.
Let $F$ be a nonempty subhypergraph of $H$, and write $E(F)=\{e_{i_1},\ldots,e_{i_e}\}$, where $i_1<\cdots<i_e$. By construction, each $e_{i_j}$ with $j\ge2$ intersects $e_{i_1}\cup\cdots\cup e_{i_{j-1}}$ in at most one vertex.
Therefore
\[
|V(F)|\ge\left|e_{i_1}\cup\cdots\cup e_{i_e}\right|
\ge r+(e-1)(r-1)=(r-1)e+1.
\]
If $e=e(F)<M$, then $\avgd(F)=\frac{re}{|V(F)|}\le\frac{re}{(r-1)e+1}=f(e)\le f(M-1)<s$ by \eqref{f(M)}. Hence every subhypergraph $F\subseteq H$ with $\avgd(F)\ge s$ satisfies $e(F)\ge M$, and thus $|V(F)|\ge(r-1)e(F)+1\ge(r-1)M+1$. This proves that $t_r(n,s,s)\ge(r-1)M+1$.
\end{proof}

\section{Critical phase: \texorpdfstring{$s=c_r$}{s = cr}}
In this section, we consider the case when $s=c_r$ and prove Theorem \ref{thm:critical}.
The following lemma characterizes subhypergraphs of average degree at least $c_r$ in terms of cycles in $B_F$.

\begin{lemma}\label{lem:incidence-cycle}
A nonempty $r$-graph $F$ contains a subhypergraph of average degree at least $c_r$ if and only if $B_F$ contains a cycle.
\end{lemma}

\begin{proof}
Suppose that $F'\subseteq F$ satisfies $\avgd(F')\ge c_r$. Then $|V(F')|\le(r-1)e(F')$. Since $B_{F'}$ has $|V(F')|+e(F')$ vertices and $re(F')$ edges, we have $e(B_{F'})\ge |V(B_{F'})|$. Thus $B_{F'}$, and hence $B_F$, contains a cycle.

Conversely, let $B_F$ contain a cycle of length $2\ell$. The cycle contains $\ell$ distinct hyperedges and $\ell$ distinct vertices of $F$. These hyperedges contain at most $(r-2)\ell$ other vertices, so together they span at most $(r-1)\ell$ vertices. Their average degree is therefore at least $r/(r-1)=c_r$.
\end{proof}
We divide the proof of Theorem \ref{thm:critical} into two cases.

\subsection{The case \texorpdfstring{$d=c_r$}{d = c(r)}}

Suppose first that $(r-1)\mid n$, and let $m=n/(r-1)$. Choose distinct vertices $x_1,\ldots,x_m$ and pairwise disjoint sets $A_1,\ldots,A_m$, each of size $r-2$ and disjoint from $\{x_1,\ldots,x_m\}$. With indices taken modulo $m$, define $e_i=\{x_i,x_{i+1}\}\cup A_i$ for $i\in[m]$, and let $H$ be a hypergraph with edges $e_1,\ldots,e_m$. Then $H$ has $(r-1)m=n$ vertices and $m$ edges, so $\avgd(H)=c_r$. By construction, $B_H$ has a unique cycle $x_1e_1x_2e_2\cdots x_me_mx_1$. 
\begin{claim}
Every subhypergraph of $H$ with average degree at least $c_r$ is $H$ itself.
\end{claim}
Let $H'\subseteq H$ satisfy $\avgd(H')\ge c_r$. By Lemma~\ref{lem:incidence-cycle}, $B_{H'}$ contains a cycle. Since $B_{H'}\subseteq B_H$ and $B_H$ has a unique cycle, the cycle in $B_H$ must be $x_1e_1x_2e_2\cdots x_me_mx_1$. These edges together cover all $n$ vertices of $H$, so $H'=H$ and $|V(H')|=n$. \hfill $\blacksquare$
\medskip

Hence $t_r(n,c_r,c_r)\ge n$. The reverse inequality is trivial, and therefore $t_r(n,c_r,c_r)=n$.
\medskip

For arbitrary $n$, write $n=(r-1)m+a$, where $0\le a\le r-2$. The case $a=0$ was proved above, so suppose that $a>0$. If $m<r-a$, then $n=O_r(1)$. Since every $r$-graph of average degree at least $c_r>1$ has at least $r+1$ vertices, we have $t_r(n,c_r,c_r)\ge r+1=\Omega_r(n)$. Together with the trivial upper bound $t_r(n,c_r,c_r)\le n$, the result follows.

We may therefore assume that $m\ge r-a$. Start with $H$ on $(r-1)m$ vertices, add $a$ new vertices, and add one more edge containing these $a$ vertices together with $r-a$ vertices chosen from $\{x_1,\ldots,x_m\}$. Choose these $r-a$ vertices as evenly as possible in the cyclic order $x_1,\ldots,x_m$, so that the cyclic distance between every two of them is at least $\lfloor m/(r-a)\rfloor$. The resulting $r$-graph has $m+1$ edges and $n$ vertices, and $\avgd(H)=r(m+1)/n\ge r/(r-1)=c_r$.

Let $C=x_1e_1x_2e_2\cdots x_me_mx_1$. Since the $a$ new vertices have degree one in $B_H$, every cycle in $B_H$ is either $C$ or consists of the vertex corresponding to the new edge together with a path on $C$ joining two of the selected vertices. Hence $g(B_H)\ge 2\lfloor m/(r-a)\rfloor+2=\Omega_r(m)=\Omega_r(n)$. If $H'\subseteq H$ satisfies $\avgd(H')\ge c_r$, then Lemma~\ref{lem:incidence-cycle} implies that $B_{H'}$ contains a cycle. Since $B_{H'}\subseteq B_H$ and half the vertices of every cycle in $B_{H'}$ belong to $V(H')$, we have $|V(H')|\ge g(B_H)/2=\Omega_r(n)$. Thus $t_r(n,c_r,c_r)=\Omega_r(n)$, and the trivial upper bound $t_r(n,c_r,c_r)\le n$ completes the case.

\subsection{The case \texorpdfstring{$d>c_r$}{d > c(r)}}

We first prove the upper bound. 
\begin{proposition}\label{prop:critical-upper}
Let $r\ge3$, $d>c_r$, and $\gamma=(r-1)d-r$. Then
$$
 t_r(n,d,c_r)\le r(r-1)\frac{\log(1+n\gamma)}{\log(1+\gamma)}.
$$
\end{proposition}

\begin{proof}
Let $H$ be an $n$-vertex $r$-graph with $m=e(H)$ and $\avgd(H)=rm/n\ge d$. Repeatedly delete any vertex of degree at most one in $B_H$, and let $C$ be the resulting graph. Let $X=V(C)\cap V(H)$ and $Y=V(C)\cap E(H)$, and write
$N=|X|$, $M=|Y|$, $E=e(C)$ and $x=E-N-M$. Then $x\ge e(B_H)-|V(B_H)|=(r-1)m-n \ge \frac{n((r-1)d-r)}{r}=\frac{n\gamma}{r}>0$ as $\gamma>0$.

Set $a=E/N$ and $b=E/M$. Since $\delta(C)\ge 2$, $E\ge2M$, and hence
$M\le N+x$. Moreover,
\begin{equation}\label{ab}
(a-1)(b-1)=\left(1+\frac{x}{N}\right)
            \left(1+\frac{x}{M}\right)
 \ge\left(1+\frac{x}{N}\right)
      \left(1+\frac{x}{N+x}\right)
=1+\frac{2x}{N}
\ge1+\frac{2\gamma}{r}.
\end{equation}

Suppose that $g(C)=2\ell$. Let $k=\lfloor(\ell-1)/2\rfloor$. By Lemma \ref{lem:bipartite-moore} and \eqref{ab},
$$
 n\ge N\ge\sum_{i=0}^{\ell-1}(b-1)^{\lceil i/2\rceil}
        (a-1)^{\lfloor i/2\rfloor}
  \ge\sum_{j=0}^{k}\left(1+\frac{2\gamma}{r}\right)^j = \frac{(1+\frac{2\gamma}{r})^{k+1}-1}{2\gamma/r}.
$$
It follows that
$$
 \ell\le2\frac{\log(1+2n\gamma/r)}
                  {\log(1+2\gamma/r)}
 \le r\frac{\log(1+n\gamma)}{\log(1+\gamma)},
$$
where the last inequality follows from $\log(1+2n\gamma/r)\le\log(1+n\gamma)$ and, by concavity of $\log(1+x)$, $\log(1+2\gamma/r)\ge(2/r)\log(1+\gamma)$.

Let $F$ be the subhypergraph formed by the $\ell$ hyperedges on this cycle. Each of these hyperedges contains two vertices of the cycle and at most $r-2$ other vertices, so $|V(F)|\le(r-1)\ell$. Hence
$\avgd(F)\ge r/(r-1)=c_r$, and therefore
\[
t_r(n,d,c_r)\le(r-1)\ell
\le r(r-1)\frac{\log(1+n\gamma)}{\log(1+\gamma)}.
\]
\end{proof}

We next prove the matching lower bound. We need two auxiliary constructions.
A \emph{theta graph} is a graph consisting of three internally vertex-disjoint paths with the same pair of distinct endpoints. It is \emph{balanced} if the lengths of these paths differ by at most one.

\begin{lemma}\label{lem:graph-excess}
There are absolute constants $a>0$ and $v_0$ such that, if $x\ge1$ and $v\ge\max\{2x,v_0\}$ are integers, then there is a simple graph $Q$ with $v$ vertices, $v+x$ edges, and $g(Q)\ge a(v/x)\log(1+x)$.
\end{lemma}
\begin{proof}
First suppose that $x=1$. Let $Q$ be a balanced theta graph on $v$ vertices and so $e(Q)=v+1$. Moreover, the union of the two shortest paths has length at least $2\lfloor (v+1)/3\rfloor$, and hence $g(Q)\ge 2\lfloor (v+1)/3\rfloor$. Thus the desired bound holds for $x=1$, provided that $v_0$ is sufficiently large and $a>0$ is sufficiently small.

Now suppose that $x\ge2$. Let $c_0>0$ be the constant from Lemma~\ref{lem:cubic-girth}, and choose $x_0\ge2$ sufficiently large so that the lemma applies to $2x$ for every $x\ge x_0$.
Let $a\le \min\{c_0/6,1/(2\log(1+x_0))\}$ be sufficiently small.

We first consider $x\ge x_0$. By Lemma~\ref{lem:cubic-girth}, there is a cubic graph $R$ on $2x$ vertices with $g(R)\ge c_0\log(2x)\ge c_0\log(1+x)$. Since $R$ has $3x$ edges, subdivide its edges as evenly as possible using exactly $v-2x$ new vertices. The resulting graph $Q$ has $v$ vertices and $e(Q)=3x+(v-2x)=v+x$. Each edge of $R$ is replaced by a path of length at least $L:=\lfloor (v+x)/(3x)\rfloor$. Since $v\ge2x$, we have $L\ge v/(6x)$. Therefore $g(Q)\ge Lg(R)\ge (c_0/6)(v/x)\log(1+x) \ge a(v/x)\log(1+x)$.

It remains to consider $2\le x<x_0$. Let $B=K_{2,x+2}$. Then $|V(B)|=x+4$, $e(B)=2x+4$, and hence $e(B)-|V(B)|=x$. Increase $v_0$, if necessary, so that $v_0\ge4x_0$. Subdivide the edges of $B$ as evenly as possible so that the resulting graph $Q$ has $v$ vertices. Since subdivision preserves $e-|V|$, we have $e(Q)=v+x$. Thus each edge of $B$ is replaced by a path of length at least $L':=\lfloor (v+x)/(2x+4)\rfloor \ge v/(8x)$. Since every cycle of $B$ has length at least $4$, we obtain $g(Q)\ge4L'\ge v/(2x)$. As $x<x_0$, we have $g(Q)\ge [1/(2\log(1+x_0))](v/x)\log(1+x)\ge a(v/x)\log(1+x)$.
\end{proof}

The second construction is based on $r$-graphs with large Berge girth. Following Berge~\cite{Berge}, a \emph{Berge cycle} of length $\ell\ge2$ in an
$r$-graph $H$ is a sequence $v_1,e_1,\ldots,v_\ell,e_\ell$ of distinct
vertices $v_1,\ldots,v_\ell$ and distinct edges $e_1,\ldots,e_\ell$ such
that $\{v_i,v_{i+1}\}\subseteq e_i$ for every $i\in[\ell]$, where indices
are taken modulo $\ell$. Equivalently, it is a cycle of length $2\ell$ in
$B_H$. The \emph{Berge girth} $\operatorname{bg}(H)$ of $H$ is the minimum length of a Berge cycle in $H$, with the convention that it is infinite if $H$ contains no Berge cycle. The following standard alteration argument will be used when the density is bounded away from $c_r$.

\begin{lemma}\label{lem:random-berge-girth}
For every fixed $r\ge3$, there is a constant $a_r>0$ such that, for all
sufficiently large $n$ and every $c_r\le d\le n^{1/2}$, there is an
$n$-vertex $r$-graph $H$ of average degree at least $d$ satisfying
\[
\operatorname{bg}(H)\ge
a_r\frac{\log(nd)}{\log(2+d)}.
\]
\end{lemma}

\begin{proof}
Let $\Delta=\binom{n-1}{r-1}$. Let $G$ be the binomial random $r$-graph on $n$ vertices with edge probability $p=4d/\Delta$ and let $Y=e(G)$. For sufficiently large $n$, we have
$p\le1$. Let $X_\ell$ denote the number of Berge cycles of length $\ell$.
Then
\[
\mathbb E X_\ell
\le n^\ell\binom{n-2}{r-2}^{\ell}p^\ell
=\left(4d\frac{n(r-1)}{n-1}\right)^\ell
\le(K_rd)^\ell
\]
for some constant $K_r$. 

Choose a constant $A_r>0$ sufficiently small and let
$L=\max\left\{2,\left\lfloor A_r\frac{\log(nd)}{\log(2+d)} \right\rfloor\right\}$ and $Z=\sum_{2\le\ell<L}X_\ell$.
Since $\log(K_rd)\le C_r\log(2+d)$ for some constant $C_r$, we may choose $A_r$ so that $A_rC_r<1/2$. If $L>2$, then
\[
\mathbb E Z
\le L(K_rd)^L
\le O_r((\log n)(nd)^{A_rC_r})
=o(dn),
\]
while $\mathbb E Z=0$ if $L=2$.

Note that $\mathbb E Y=\binom{n}{r}p=4dn/r$, we have $\mathbb E(Y-Z)=\frac{4dn}{r}-o(dn)>\frac{dn}{r}$. Thus some realization satisfies $Y-Z\ge dn/r$. Choose one edge from each Berge cycle of length less than $L$ and delete all chosen edges. The resulting $r$-graph $H$ has at least $dn/r$ edges and Berge girth at least $L$. Hence $\avgd(H)\ge d$, and
\[
\operatorname{bg}(H)\ge L
\ge\frac{A_r}{2}\frac{\log(nd)}{\log(2+d)}.
\]
Taking $a_r=A_r/2$ completes the proof.
\end{proof}

We are now ready to prove the lower bound.

\begin{proposition}\label{prop:critical-lower}
For every $r\ge3$, there is a constant $b_r>0$ such that, for every $n$ and every $d>c_r$, we have
$$
t_r(n,d,c_r)\ge
b_r\frac{\log(1+n\gamma)}{\log(1+\gamma)},
$$
where $\gamma=(r-1)d-r$.
\end{proposition}

\begin{proof}
Since $1+n\gamma\le(1+\gamma)^n$, we have
$\log(1+n\gamma)/\log(1+\gamma)\le n$. Moreover, $t_r(n,d,c_r)\ge r+1$ since $d>c_r$. Thus, by decreasing $b_r$ if necessary, we may assume that $n\ge n_0(r)$ for some sufficiently large constant $n_0(r)$.

We first consider the case $d>n^{1/2}$. In this case $n<d^2$ and $\gamma=\Theta_r(d)$, so
$
\frac{\log(1+n\gamma)}{\log(1+\gamma)}=\Theta_r\left(\frac{\log(1+nd)}{\log(1+d)}\right)=O_r(1).
$
The required bound follows from $t_r(n,d,c_r)\ge r+1$ by decreasing
$b_r$. 

Hence we may assume that $d\le n^{1/2}$. Fix a sufficiently small constant $\gamma_0=\gamma_0(r)>0$. Suppose first that $0<\gamma\le\gamma_0$. 
Choose an integer $x$ such that
$x\equiv -n\pmod{r-1}$ and
$\frac{n\gamma}{r}\le x<\frac{n\gamma}{r}+r-1$.
Let $m=\frac{n+x}{r-1}$ and $v=m-x=\frac{n-(r-2)x}{r-1}$. For sufficiently small $\gamma_0$, we have $v\ge2x$ and $v=\Theta_r(n)$. By Lemma~\ref{lem:graph-excess}, there is a graph $Q$ with $v$ vertices, $v+x=m$ edges, and $g(Q)\ge a\frac{v}{x}\log(1+x)$.
For each edge $uv\in E(Q)$, let $A_{uv}$ be a set of $r-2$ new vertices, where these sets are pairwise disjoint, and define $E(H)=\bigl\{\{u,v\}\cup A_{uv}:uv\in E(Q)\bigr\}$. The resulting $r$-graph $H$ has $m$ edges and $|V(H)|=v+(r-2)m=n$. Since $(r-1)m-n=x$, we have $(r-1)\avgd(H)-r=\frac{rx}{n}\ge\gamma$, and hence $\avgd(H)\ge d$. Every vertex of $V(H)\setminus V(Q)$ has degree one in $B_H$ and hence lies on no cycle. Therefore every cycle in $B_H$ corresponds to a cycle in $Q$. If $F\subseteq H$ has average degree at least $c_r$, then Lemma~\ref{lem:incidence-cycle} gives a cycle in $B_F$. Its vertices on the $V(H)$-side form a cycle in $Q$, so $|V(F)|\ge g(Q)$. Since $x=\Theta_r(1+n\gamma)$, $v=\Theta_r(n)$, and $\log(1+\gamma)=\Theta(\gamma)$ in this range, we obtain
$
|V(F)|\ge  g(Q)\ge b_r\frac{\log(1+n\gamma)}{\log(1+\gamma)}.
$

It remains to consider $\gamma\ge\gamma_0$. Since $d\le n^{1/2}$, Lemma~\ref{lem:random-berge-girth} gives an $n$-vertex $r$-graph $H$ of average degree at least $d$ satisfying
$
\operatorname{bg}(H)\ge a_r\frac{\log(nd)}{\log(2+d)}=\Omega_r\left(\frac{\log(1+n\gamma)}{\log(1+\gamma)}\right),
$
where we used $1+\gamma=\Theta_r(1+d)$ and $1+n\gamma=\Theta_r(1+nd)$. Every subhypergraph of $H$ of average degree at least $c_r$ contains a Berge cycle and therefore has at least $\operatorname{bg}(H)$ vertices. This proves the proposition.
\end{proof}

Combining the upper and lower bounds in Propositions~\ref{prop:critical-upper} and~\ref{prop:critical-lower} with the case $d=c_r$ above proves Theorem~\ref{thm:critical}.

\section{Supercritical phase: \texorpdfstring{$s>c_r$}{s > cr}}
\subsection{Lower bound via alteration}
We begin with the proof of Theorem~\ref{thm:lower}. The proof is an alteration argument. Recall that $\alpha_{r,s}=s/((r-1)s-r)$, $w=nd^{-\alpha_{r,s}}$ and $\Lambda=\log(2+nd)$.

\begin{proof}[Proof of Theorem~\ref{thm:lower}] 
Let $K=\left\lfloor c\left(w+\frac{\Lambda}{\log(2+\Lambda/w)}\right)\right\rfloor$, where $c=c(r,s)>0$ will be chosen sufficiently small. Let $G$ be the binomial random $r$-graph on $n$ vertices obtained by including each edge independently with probability $p=2d/\binom{n-1}{r-1}$. Since $d\le n^{(r-1)-r/s}$, we have $p=o(1)$. Moreover $\mathbb E e(G)=2dn/r$. 

For each integer $v\ge r$, let $m_v=\lceil sv/r\rceil$, and let $Z_v$ be the number of pairs $(X,\mathcal F)$ such that $X\subseteq V(G)$, $|X|=v$, $\mathcal F\subseteq\binom Xr$, $|\mathcal F|=m_v$, and $\mathcal F\subseteq E(G)$. If $m_v>\binom vr$, then $Z_v=0$. 
If $m_v\le\binom vr$, then, using $\binom nv\le(en/v)^v$, $\binom vr\le v^r/r!$ and $\binom{n-1}{r-1}\ge(n/2)^{r-1}/(r-1)!$, we obtain 
\begin{align} 
\mathbb EZ_v &=\binom nv\binom{\binom vr}{m_v}\left(\frac{2d}{\binom{n-1}{r-1}}\right)^{m_v}\notag\\ &\le\left(\frac{en}{v}\right)^v\left(\frac{e\binom vr}{m_v}\frac{2d}{\binom{n-1}{r-1}}\right)^{m_v}\notag\\ &\le\left(\frac{en}{v}\right)^v\left(\frac{e(v^r/r!)}{sv/r}\frac{2d}{(n/2)^{r-1}/(r-1)!}\right)^{m_v}\notag\\ &=\left(\frac{en}{v}\right)^v\left(\frac{2^re}{s}d\left(\frac vn\right)^{r-1}\right)^{m_v}\notag\\ &\le\left(\frac{en}{v}\right)^v\left(C_{r,s}d\left(\frac vn\right)^{r-1}\right)^{m_v}, \label{eq:configuration-first} 
\end{align}
where $C_{r,s}\ge\frac{2^re}{s}$.
We prove the following claim. 
\medskip
\begin{claim}\label{c1}
If $c=c(r,s)>0$ is sufficiently small and $n$ is sufficiently large, then \[ C_{r,s}d\left(\frac vn\right)^{r-1}\le1 \] for every $r\le v\le K$. 
\end{claim} 
Suppose first that $w\ge\Lambda$. Then $\Lambda/\log(2+\Lambda/w)=O(w)$, so $K=O(cw)$. Consequently, for every $v\le K$, there is a constant $C'_{r,s}>0$ such that $C_{r,s}d\left(\frac vn\right)^{r-1}\le C_{r,s}d\left(\frac Kn\right)^{r-1}\le C'_{r,s}c^{r-1}d\left(\frac wn\right)^{r-1}=C'_{r,s}c^{r-1}d^{-r/((r-1)s-r)}<1$ when $c$ is sufficiently small. Now suppose that $w<\Lambda$. Since $\Lambda/w>1$, we have $w \le \frac{2\Lambda}{\log(2+\Lambda/w)}$, and hence $K=O(c\Lambda)$. As $d\le n^{(r-1)-r/s}$, there is a constant $C'_{r,s}>0$ such that \[ C_{r,s}d\left(\frac vn\right)^{r-1}\le C_{r,s}d\left(\frac Kn\right)^{r-1}\le C'_{r,s}d\left(\frac{\Lambda}{n}\right)^{r-1}\le C'_{r,s}\Lambda^{r-1}n^{-r/s}=o(1). \] This proves the claim. 
\hfill $\blacksquare$ 

\medskip 

By Claim~\ref{c1} and \eqref{eq:configuration-first}, since
$m_v\ge sv/r$, we have, for every $r\le v\le K$,
\begin{align}
\mathbb EZ_v
&\le
\left(\frac{en}{v}\right)^v
\left(C_{r,s}d\left(\frac vn\right)^{r-1}\right)^{sv/r}
\notag\\
&=
\left[
C'_{r,s}d^{s/r}
\left(\frac vn\right)^{((r-1)s-r)/r}
\right]^v
\notag\\
&=
\left[
C'_{r,s}
\left(\frac vw\right)^{((r-1)s-r)/r}
\right]^v,
\label{eq:configuration-count}
\end{align}
where $C'_{r,s}:=eC_{r,s}^{s/r}$.

Let $Z=\sum_{r\le v\le K}Z_v$. If $w\ge\Lambda$, then $K=O(cw)$, so \eqref{eq:configuration-count} gives $\mathbb EZ_v\le\left[C'_{r,s}\left(\frac Kw\right)^{((r-1)s-r)/r}\right]^v\le2^{-v}$ when $c$ is sufficiently small. Hence $\mathbb EZ=O_{r,s}(1)=o(nd)$. Now suppose that $w<\Lambda$. 
Since $K=O\left(c\Lambda/\log(2+\Lambda/w)\right)$, we have $K/w=O(\Lambda/w)$, and hence $\log(2+K/w)=O(\log(2+\Lambda/w))$.
Thus, for some constants $C''_{r,s},C'''_{r,s}>0$ and every $v\le K$,
\[
\begin{aligned}
\log \mathbb EZ_v
&\le v\left(\log C'_{r,s}
+\frac{(r-1)s-r}{r}\log\frac vw\right)\\
&\le C''_{r,s}v\log\left(2+\frac Kw\right)\\
&\le C''_{r,s}K\log\left(2+\frac Kw\right)\\
&\le C'''_{r,s}c\Lambda.
\end{aligned}
\]
Since $c$ is sufficiently small, we obtain $\mathbb EZ\le K\exp(\Lambda/4)\le C_{r,s}\Lambda(2+nd)^{1/4}=o(nd)$. Therefore $\mathbb EZ=o(nd)$ in both cases, and hence \[ \mathbb E\bigl(e(G)-Z\bigr)=\frac{2dn}{r}-o(nd)>\frac{dn}{r}. \] Fix a realization for which $e(G)-Z>dn/r$. For every configuration counted by $Z$, select one of its edges, and let $H$ be obtained by deleting all selected edges. Then $e(H)\ge e(G)-Z>dn/r$, so $\avgd(H)\ge d$. Moreover, $H$ contains no subhypergraph of average degree at least $s$ on $v\le K$ vertices. Consequently, \[ t_r(n,d,s)\ge K+1>c\left(w+\frac{\Lambda}{\log(2+\Lambda/w)}\right), \] as required.
\end{proof}


\subsection{Upper bound via uniformity reduction}
The following dichotomy is the main reduction used in the proof of Theorem~\ref{thm:main}. We distinguish the $(r-1)$-sets of large codegree. If many edges contain such a set, Hall's theorem gives a small dense subhypergraph. Otherwise, after deleting these edges, all $(r-1)$-sets have bounded codegree, and sampling a small set of vertices produces a dense $(r-1)$-graph. The sampling step is related to the random-marking argument of Feige and Wagner~\cite{FW}.

\begin{lemma}\label{lem:compression} 
Fix $r\ge3$ and $\sigma>1$. There are constants $c,C,d_0>0$, depending only on $r$ and $\sigma$, such that the following holds. Let $H$ be an $n$-vertex $r$-graph of average degree at least $d\ge d_0$, and let $0<p\le1/(16r)$ satisfy $pn\ge r$. Then one of the following holds. 
\begin{enumerate}[label=\textnormal{(\alph*)}] 
\item $H$ contains a subhypergraph of average degree at least $\sigma$ and order at most $Cpn$. 
\item There exist a set $X\subseteq V(H)$ with $|X|=\lceil pn\rceil$ and an $(r-1)$-graph $G$ on $V(H)\setminus X$ with $\avgd(G)\ge cpd$ such that, for every $P\in E(G)$, there exists $x_P\in X$ with $P\cup\{x_P\}\in E(H)$.
\end{enumerate} 
\end{lemma} 
\begin{proof} 
Write $k=\lceil pn\rceil$ and $h=\lceil 2\sigma(r-1)/r\rceil$, and set $A=r(2h+2)$.
Call an $(r-1)$-set $P$ \emph{heavy} if $d_H(P)\ge A/p$, and let $\mathcal E=\{e\in E(H): e\text{ contains a heavy $(r-1)$-set}\}$.

\medskip
\noindent
\textbf{Case 1. $|\mathcal E|\ge e(H)/2$.}
\medskip

For every heavy $P$, take $\lfloor d_H(P)/r\rfloor$ copies of $P$. Construct a bipartite graph with these copies and $E(H)$ as two parts, and join a copy of $P$ to every edge containing $P$. Consider any set $\mathcal S$ of copies, and let $m_P$ be the number of copies of $P$ in $\mathcal S$. Then
\[
|\mathcal S|
=\sum_P m_P
\le
\sum_{P:m_P>0}\left\lfloor\frac{d_H(P)}r\right\rfloor
\le
\frac1r\sum_{P:m_P>0}d_H(P).
\]
Each edge of $H$ contains at most $r$ sets $P$ with $m_P>0$, and hence
$
\sum_{P:m_P>0}d_H(P)
\le r|N(\mathcal S)|.
$
Thus $|\mathcal S|\le |N(\mathcal S)|$ and Hall's theorem holds. Hence we may assign to every heavy $P$ exactly $\lfloor d_H(P)/r\rfloor$ distinct edges containing $P$, with no edge assigned to two different sets. Since $d_H(P)\ge A/p\ge r$ for every heavy $P$,
$
\left\lfloor\frac{d_H(P)}r\right\rfloor
\ge \frac{d_H(P)}{2r}.
$
Therefore, using $|\mathcal E|\ge e(H)/2$,
\[
\sum_{P\text{ heavy}}
\left\lfloor\frac{d_H(P)}r\right\rfloor
\ge \frac1{2r}\sum_{P\text{ heavy}}d_H(P)
\ge \frac{|\mathcal E|}{2r}
\ge \frac{e(H)}{4r}
\ge \frac{dn}{4r^2}.
\]

Choose a uniformly random $k$-set $X$. For each heavy $P$, let $Z_P$ be the number of edges assigned to $P$ whose vertex outside $P$ belongs to $X$. Hence 
\begin{equation}\label{Zp}
\E Z_P =\frac{k}{n}\left\lfloor\frac{d_H(P)}r\right\rfloor \ge p\left(\frac{A}{rp}-1\right) \ge 2h+1.
\end{equation}
Call a heavy $(r-1)$-set $P$ \emph{good} if $Z_P\ge \E Z_P/2$. Since
$Z_P<\E Z_P/2$ whenever $P$ is not good,
$
\E[Z_P\mathbf 1_{\{P\text{ good}\}}]\ge \E Z_P/2.$ It follows that some choice of $X$ satisfies 
$
 \sum_{P\text{ good}}Z_P \ge \frac{k}{2n} \sum_{P\text{ heavy}} \left\lfloor\frac{d_H(P)}r\right\rfloor \ge \frac{dk}{8r^2}.    
$
Set $M=\lceil3\sigma(k+r-1)/r\rceil$. Since $k\ge r$, we have $M=O_{r,\sigma}(k)$, so $\frac{dk}{8r^2}\geq M$ when $d_0$ is sufficiently large. Order the good sets as $P_1,P_2,\ldots$ such that $\sum_{i=1}^{t-1} Z_{P_i}<M\le \sum_{i=1}^t Z_{P_i}$ with $t$ minimal. Select all the edges counted by $Z_{P_1},\ldots,Z_{P_{t-1}}$, together with enough of those counted by $Z_{P_t}$ to obtain exactly $M$ edges. Since each $P_i$ is good, we have $Z_{P_i}\ge h$ for every $i<t$ by \eqref{Zp}.
Note that $(t-1)h<M$, and we have $t\le M/h+1$. Since $|X|=k$, $|P_i|=r-1$, and $t\le M/h+1$, the selected edges span at most $k+(r-1)\left(\frac Mh+1\right)$ vertices. Let $F$ be the subhypergraph formed by the selected edges. Then $e(F)=M$ and $|V(F)|\le k+(r-1)\left(\frac Mh+1\right)$. Hence $r e(F)-\sigma|V(F)|\ge rM-\sigma\left(k+(r-1)\left(\frac Mh+1\right)\right)=\left(r-\frac{\sigma(r-1)}h\right)M-\sigma(k+r-1)\ge\frac{rM}{2}-\sigma(k+r-1)>0$. Therefore $\avgd(F)=\frac{r e(F)}{|V(F)|}>\sigma$. Moreover, $|V(F)|=O_{r,\sigma}(k)=O_{r,\sigma}(pn)$, so (a) follows by choosing $C$ sufficiently large.

\medskip
\noindent
\textbf{Case 2. $|\mathcal E|< e(H)/2$.}
\medskip

Let $H_0=H-\mathcal E$. Then $e(H_0)\ge dn/(2r)$ and $d_{H_0}(P)<\frac{2r(h+1)}p$ for every $(r-1)$-set $P$. For a uniformly random $k$-set $X$, let $G_X$ be the $(r-1)$-graph on $V(H)\setminus X$ in which an $(r-1)$-set is an edge if it can be extended to an edge of $H_0$ by some vertex of $X$. For every fixed $(r-1)$-set $P\subseteq V(H)$, since $k\le2pn$, 
\begin{equation}\label{P}
\Pr(P\cap X\ne\varnothing)\le\frac{(r-1)k}{n}\le2(r-1)p<\frac12. 
\end{equation}
Conditioned on $P\cap X=\varnothing$, the set $X$ is a uniformly random $k$-subset of $V(H)\setminus P$. Therefore,
\[
\Pr\bigl(X\cap N_{H_0}(P)=\varnothing\mid P\cap X=\varnothing\bigr)
=
\frac{\binom{n-r+1-d_{H_0}(P)}{k}}
     {\binom{n-r+1}{k}}
\le
\left(1-\frac{d_{H_0}(P)}{n-r+1}\right)^k
\le
\exp\left(-\frac{k\,d_{H_0}(P)}{n-r+1}\right),
\]
where $N_{H_0}(P)=\{x\in V(H)\setminus P:P\cup\{x\}\in E(H_0)\}$.
Moreover, $p\,d_{H_0}(P) \le \frac{k\,d_{H_0}(P)}{n-r+1} \le 4A$. Since $1-e^{-x}\ge c_0x$ for $0\le x\le4A$, where $c_0=c_0(r,\sigma)>0$, we have
\[
\Pr\bigl(X\cap N_{H_0}(P)\ne\varnothing
\mid P\cap X=\varnothing\bigr)
\ge c_0p\,d_{H_0}(P).
\]
Together with \eqref{P}, this gives $\Pr(P\in E(G_X))\ge c_1 p\,d_{H_0}(P)$ for some $c_1=c_1(r,\sigma)>0$. Consequently,
$
\E e(G_X)
\ge c_1p\sum_P d_{H_0}(P)
=c_1pr e(H_0)
\ge c_2pdn
$
for some $c_2=c_2(r,\sigma)>0$. Hence there is a choice of $X$ such that $\avgd(G_X)=\frac{(r-1)e(G_X)}{n-k}\ge c_3pd$ for some constant $c_3=c_3(r,\sigma)>0$. Let $G=G_X$ for such a choice of $X$. For every $P\in E(G)$, choose $x_P\in X$ such that $P\cup\{x_P\}\in E(H_0)$. 
Thus (b) holds. 
\end{proof}

The following result shows that any prescribed amount of excess can be found in a small subhypergraph.

\begin{lemma}\label{lem:excess-localization}
Fix $r\ge2$. There are constants $c,C,d_0>0$, depending only on $r$, such that the following holds. Let $H$ be an $n$-vertex $r$-graph of average degree at least $d\ge d_0$, and let $1\le a\le cdn$ be an integer. Then $H$ contains a subhypergraph $F$ such that $\xi_r(F)\ge a$ and $|V(F)|\le Ca\log\left(2+\frac{dn}{a}\right)$.
\end{lemma}

\begin{proof}
We first prove the following claim.
\begin{claim}\label{clm:graph-excess}
There is an absolute constant $C>0$ such that the following holds. Let $J$ be a graph with minimum degree at least three, and let $1\le a\le e(J)-|V(J)|$. Then $J$ contains a subgraph $Q$ such that $e(Q)-|V(Q)|\ge a$ and $|V(Q)|\le Ca\log\left(2+\frac{e(J)-|V(J)|}{a}\right)$.
\end{claim}

Assume first that $J$ is connected. Since $\delta(J)\ge3$,
$
e(J)-|V(J)|\ge \frac{|V(J)|}{2}.
$
If $|V(J)|\le4a$, we may take $Q=J$. Otherwise, choose a connected set $S\subseteq V(J)$ with $|S|=4a$. If
$
e(J[S])-|S|\ge a,
$
we are done. Otherwise, the number of edges between $S$ and $V(J)\setminus S$ is at least
$
3|S|-2e(J[S])
=
|S|-2\bigl(e(J[S])-|S|\bigr)
\ge2a+2.
$
Contract $S$ to a single vertex $z$ and delete the resulting loops while keeping parallel edges. Every vertex other than $z$ has degree at least three, while $d(z)\ge2a+2$. A breadth-first search from $z$ finds a cycle $C$ within $O\left(\log\left(2+\frac{|V(J)|}{a}\right)\right)$ levels. Let $P$ be the path in the breadth-first-search tree from $z$ to the vertex of $C$ closest to $z$. Then $C\cup P$ is connected and unicyclic, contains $z$, and has $O(\log(2+|V(J)|/a))$ vertices. Replacing $z$ by $S$ therefore adds at least one more edge than vertex, increasing $e(J[S])-|S|$ by at least one. Initially $J[S]$ is connected, so $e(J[S])-|S|\ge-1$, and each iteration increases this by at least one. Thus after at most $a+1$ iterations we obtain a set $S$ such that $e(J[S])-|S|\ge a$. Since each iteration adds $O(\log(2+|V(J)|/a))$ vertices, the subgraph $Q=J[S]$ satisfies $|V(Q)|=O\left(a\log\left(2+\frac{|V(J)|}{a}\right)\right)$. Finally, $\delta(J)\ge3$ implies $|V(J)|\le2(e(J)-|V(J)|)$, and the connected case follows.

Now suppose that $J$ is disconnected. Every component $J_i$ satisfies $e(J_i)-|V(J_i)|\ge|V(J_i)|/2>0$. If some component has excess at least $a$, apply the connected case to that component. Otherwise, choose components successively until their total excess first reaches $a$. Their total excess is less than $2a$, and their union has fewer than $4a$ vertices. This proves the claim.
\hfill $\blacksquare$
\medskip

We now return to $H$. By deleting edges, we may assume that $\frac{dn}{r}\le e(H)<\frac{dn}{r}+1.$ Suppose first that $r=2$. Repeatedly delete vertices of degree at most two, and let $J$ be the remaining graph. Such deletion does not decrease $e-2v$, so for $d_0$ sufficiently large, $e(J)-2|V(J)|\ge e(H)-2n\ge \frac{dn}{4}$. Taking $c\le 1/4$, we have $a\le cdn\le \frac{dn}{4}\le e(J)-|V(J)|$. Thus, applying Claim~\ref{clm:graph-excess} to $J$ gives a subgraph $F$ satisfying $\xi_2(F)=e(F)-|V(F)|\ge a$
and $|V(F)|\le Ca\log\left(2+\frac{dn}{a}\right)$.

Now let $r\ge3$. Consider the incidence graph $B_H$. Hence, for $d_0$ sufficiently large,
$
e(B_H)-2|V(B_H)|
=(r-2)e(H)-2n
\ge \frac{dn}{2r}.
$
Repeatedly delete vertices of degree at most two from $B_H$, and let $J$ be the remaining graph. Thus,
$
e(J)-|V(J)|
\ge e(J)-2|V(J)|
\ge \frac{dn}{2r},
$
and we may choose $c=c(r)>0$ so that $a\le cdn\le e(J)-|V(J)|$.
Applying Claim~\ref{clm:graph-excess} to $J$, we obtain a subgraph
$Q\subseteq J$ with $e(Q)-|V(Q)|\ge a$.
Note that
$
e(J)-|V(J)|
\le e(B_H)
= re(H)
=O_r(dn),
$
and we have
$
|V(Q)|
\le
Ca\log\left(2+\frac{dn}{a}\right).
$
Delete all isolated vertices from $Q$; this does not decrease $e(Q)-|V(Q)|$. Let $\mathcal F=V(Q)\cap E(H)$, and let $F$ be the subhypergraph of $H$ with edge set $\mathcal F$ and vertex set $\bigcup_{e\in\mathcal F}e$. Its incidence graph $B_F$ is obtained from $Q$ by adding missing vertices together with a first incident edge, and then adding the remaining incidences. Consequently, $ \xi_r(F)=e(B_F)-|V(B_F)|\ge e(Q)-|V(Q)|\ge a. $ Finally, $e(F)=|\mathcal F|\le|V(Q)|$, and hence $ |V(F)|\le r e(F)\le r|V(Q)|\le C_ra\log\left(2+\frac{dn}{a}\right). $ Increasing $C=C(r)$ completes the proof. 
\end{proof}

We will use the following uniform version of a result of Janzer, Sudakov and Tomon \cite{JST}, allowing the target average degree to vary in a fixed neighborhood of $s_0$ while the constants are uniform.
\begin{lemma}\label{lem:JST-uniform}
Fix $s_0>2$. There are constants $\eta,C,d_0>0$, depending only on $s_0$,
such that the following holds. Let $|s-s_0|\le\eta$ and
$
d_0\le d\le N^{(s-2)/s}.
$
Every $N$-vertex graph of average degree at least $d$ contains a nonempty
subgraph $F$ of average degree at least $s$ such that
\[
|V(F)|\le N d^{-s/(s-2)}(\log d)^C.
\]
\end{lemma}

We defer the proof of Lemma \ref{lem:JST-uniform} to Appendix~\ref{JST}.

\medskip

Now we establish a stable version of the supercritical upper bound, in which the target average degree is allowed to vary slightly around a fixed value $u_0$.


\begin{proposition}\label{prop:stable-supercritical}
Fix an integer $r\ge2$ and a real number $u_0>c_r=r/(r-1)$. There are constants $A,C,d_0>0$, depending only on $r$ and $u_0$, such that the following holds. Let $d\ge d_0$ and let $u>c_r$ satisfy $|u-u_0|\le(\log d)^{-A}$ and $d\le n^{1/\alpha_{r,u}}$. Every $n$-vertex $r$-graph $G$ of average degree at least $d$ contains a nonempty subhypergraph $F$ of average degree at least $u$ such that
\[
|V(F)|\le nd^{-\alpha_{r,u}}(\log d)^C.
\]
\end{proposition}

\begin{proof}
We argue by induction on $r$. For $r=2$, the result follows from Lemma~\ref{lem:JST-uniform}, since $\alpha_{2,u}=u/(u-2)$.
Assume $r\ge3$ and that the result holds for $(r-1)$-graphs. 
Write $\ell:=\log d$.
We divide the proof into three cases. In each case below, we first choose $A > 2$ and then choose $d_0$ sufficiently
large.

\medskip
\noindent\textbf{Case 1: $u_0<r/(r-2)$.}
\medskip

Since $|u-u_0|\le \ell^{-A}=o(1)$, we may assume that $u<r/(r-2)$. Set $q_0:=2u_0/(r-(r-2)u_0)$ and $q:=2u/(r-(r-2)u)$. Then $q_0>2$. Let $\eta>0$ be the constant supplied by Lemma~\ref{lem:JST-uniform} with $s_0=q_0$. By continuity, we have $|q-q_0|\le\eta$. Moreover, $q/(q-2)=\alpha_{r,u}$ and $rq/(2+(r-2)q)=u$. 
Let $M:=\lceil q\rceil$. If some pair has codegree at least $M$, then any $M$ such edges form a subhypergraph on at most $2+(r-2)M$ vertices and of average degree at least $rM/(2+(r-2)M)\ge u$, and we are done.

Hence we may assume that every pair has codegree at most $M-1$. For each $e\in E(G)$, choose an arbitrary pair
$p(e)\in\binom{e}{2}$, and let $J$ be the graph whose edges are the distinct pairs $p(e)$. Since every pair is contained in at most $M-1$ edges of $G$, we have $e(J)\ge e(G)/(M-1)$. Since $M=O_{r,u_0}(1)$, there is a constant $c_0=c_0(r,u_0)>0$ such that $\avgd(J)\ge c_0d$. We may assume that $c_0\le1$. Since $\alpha_{r,u}=q/(q-2)$ and $d\le n^{1/\alpha_{r,u}}$, we have $c_0d\le n^{(q-2)/q}$. Applying Lemma~\ref{lem:JST-uniform} with $s_0=q_0$, $s=q$ and density parameter $c_0d$, we obtain a subgraph $J'$ of average degree at least $q$ such that $|V(J')|\le n(c_0d)^{-q/(q-2)}(\log(c_0d))^C\le nd^{-\alpha_{r,u}}(\log d)^{O_{r,u_0}(1)}$. Let $t:=|V(J')|$ and $m:=\lceil qt/2\rceil$, and choose $m$ edges of $J'$. For each chosen edge $xy$, select an edge $e_{xy}\in E(G)$ such that $p(e_{xy})=xy$, and let $F$ be the subhypergraph of $G$ consisting of these edges. Then $|V(F)|\le t+(r-2)m=O_{r,u_0}(t)$ and $\avgd(F)\ge rm/(t+(r-2)m)\ge rq/(2+(r-2)q)=u$, as required.

\medskip

We shall use the following claim in the remaining two cases.
\begin{claim}\label{clm:injective-lifting}
Let $X\subseteq V(G)$ have size $k$, and let $J$ be an $(r-1)$-graph on $V(G)\setminus X$. Suppose that for every $e\in E(J)$ there is $x_e\in X$ such that $e\cup\{x_e\}\in E(G)$. If $W\subseteq V(J)$ satisfies $re_J(W)\ge u(|W|+k)$, then $G$ contains a subhypergraph of average degree at least $u$ and order at most $|W|+k$.
\end{claim}
Indeed, let $F$ be the subhypergraph of $G$ consisting of the edges $e\cup\{x_e\}$ with $e\in E(J[W])$. Since $e\subseteq V(G)\setminus X$ and $x_e\in X$, these lifted edges are pairwise distinct. Hence $e(F)=e_J(W)$, while $V(F)\subseteq W\cup X$. Therefore,
$
\avgd(F)=\frac{re(F)}{|V(F)|}\ge \frac{re_J(W)}{|W|+k}\ge u.
$
\hfill $\blacksquare$

\medskip
\noindent\textbf{Case 2: $u_0=r/(r-2)$.}
\medskip

Set $p:=\ell^4d^{-\alpha_{r,u}}$ and $k:=\lceil pn\rceil$. Since $d\le n^{1/\alpha_{r,u}}$, we have $pn\ge\ell^4$. Since $\alpha_{r,u_0}=1$ and $|u-u_0|\le\ell^{-A}$, we have $\alpha_{r,u}=1+O_r(\ell^{-A})$. Thus, $pd=\ell^4d^{1-\alpha_{r,u}}=\Theta_r(\ell^4)$ and, in particular, $p=o(1)$. For sufficiently large $d$, we have $p\le1/(16r)$ and $pn\ge r$. We may therefore apply Lemma~\ref{lem:compression} with $\sigma=u_0+1$. Since $u=u_0+o(1)<u_0+1$, (a) gives a subhypergraph of average degree at least $u$ on at most $Cpn=Cn\ell^4d^{-\alpha_{r,u}}$ vertices, and we are done.

Now suppose (b) holds. Then there are a set $X$ of size $k$ and an $(r-1)$-graph $J$ on $V(G)\setminus X$ with average degree at least $c_1pd$, where $c_1=c_1(r,u_0)>0$. Moreover, for every $e\in E(J)$ there is $x_e\in X$ such that $e\cup\{x_e\}\in E(G)$. We have $c_1pd(n-k)/(2k)=(1+o(1))c_1d/2\to\infty$. Hence, Lemma~\ref{lem:excess-localization}, applied to $J$ with density parameter $c_1pd$ and $a=2k$, gives a subhypergraph $Q\subseteq J$ such that $(r-2)e(Q)-|V(Q)|\ge 2k$ and $|V(Q)|\le C2k\log(2+c_1pd(n-k)/(2k))=O_{r,u_0}(k\ell)$. Since $u_0(r-2)=r$, we have $re(Q)-u(|V(Q)|+k)=u_0((r-2)e(Q)-|V(Q)|)-(u-u_0)|V(Q)|-uk$. Moreover, $|u-u_0||V(Q)|=O_{r,u_0}(k\ell^{1-A})=o(k)$. Hence $re(Q)-u(|V(Q)|+k)\ge 2u_0k-uk-o(k)$. We obtain $re(Q)\ge u(|V(Q)|+k)$. We apply Claim~\ref{clm:injective-lifting} with $W=V(Q)$, which yields a subhypergraph $F$ of average degree at least $u$ and order at most $|V(Q)|+k=O_{r,u_0}(k\ell)\le nd^{-\alpha_{r,u}}\ell^{O_{r,u_0}(1)}$, as required.

\medskip
\noindent\textbf{Case 3: $u_0>r/(r-2)$.}
\medskip

Set $q:=(r-1)u/r+\ell^{-A/2}$ and $b:=q/((r-1)(q-1))$. Let $A'$ be the constant supplied by the induction hypothesis for $(r-1)$-graphs with fixed target $q_0=(r-1)u_0/r$, and choose $A>2A'+2$. A direct calculation gives $0\le\alpha_{r,u}-b=O_{r,u_0}(\ell^{-A/2})$ and $(1-b)\alpha_{r-1,q}=b$. Hence $d^{\alpha_{r,u}-b}=\exp\bigl(O_{r,u_0}(\ell^{1-A/2})\bigr)=1+o(1)$. Set $p:=\ell^4d^{-b}$ and $k:=\lceil pn\rceil$. Since $b\le\alpha_{r,u}$ and $d\le n^{1/\alpha_{r,u}}$, we have $pn=\ell^4nd^{-b}\ge\ell^4$. Let $B:=\lceil (r-1)u\ell^{A/2}/r\rceil$. Then $B=\ell^{O_{r,u_0}(1)}$,
$rqB/((r-1)(B+1))\ge u$. Since
$
q=\frac{(r-1)u_0}{r}+o(1),
$
we have $b=\alpha_{r,u_0}+o(1)$. Hence
$b\ge\alpha_{r,u_0}/2$ for sufficiently large $d$, and therefore
$
pB
=O_{r,u_0}\left(\ell^{4+A/2}
d^{-\alpha_{r,u_0}/2}\right)
=o(1).
$ 
Since $pn\to\infty$,
we have $k=(1+o(1))pn$, so $(B+1)k/n=o(1)$. Hence $Bk<n-k$. Apply Lemma~\ref{lem:compression} with $\sigma=u_0+1$. Since $u<u_0+1$ for sufficiently large $d$, if (a) is true, then there exists a subhypergraph of average degree at least $u$ on $O_{r,u_0}(pn)\le nd^{-\alpha_{r,u}}\ell^{O_{r,u_0}(1)}$ vertices, as required.

Now suppose (b) holds. Then there are a set $X$ of size $k$ and an $(r-1)$-graph $J$ on $V(G)\setminus X$ with average degree at least $cpd$, where $c=c(r,u_0)>0$, such that for every $e\in E(J)$ there is $x_e\in X$ with $e\cup\{x_e\}\in E(G)$. We construct $W\subseteq V(J)$ iteratively, starting with $W=\varnothing$. At any stage, if $re_J(W)\ge u(|W|+k)$, then apply Claim~\ref{clm:injective-lifting} and we are done. Otherwise, as long as $|W|<Bk$, we have $e_J(W)=O_{r,u_0}(Bk)$. 
For the current set $W$, define $J_W:=J-E(J[W])$. Then $J_W$ has average degree at least $c'pd$ for some $c'=c'(r,u_0)>0$, since $e(J)=\Omega_{r,u_0}(pdn)$ and $Bk=O_{r,u_0}(Bpn)=o(pdn)$.
Since $q=(r-1)u_0/r+o(1)>c_{r-1}$, we have
$\alpha_{r-1,q}=\Theta_{r,u_0}(1)$. Moreover, using
$(1-b)\alpha_{r-1,q}=b$, $d^b\le n$, and $k=o(n)$, we obtain
$
\left(\frac{d^{1-b}}2\right)^{\alpha_{r-1,q}}
=2^{-\alpha_{r-1,q}}d^b
\le n-k
$
for sufficiently large $d$. We also have $d^{1-b}/2\le c'pd$ and $\log(d^{1-b}/2)=\Theta_{r,u_0}(\ell)$. Moreover, for sufficiently large $d$,
$
|q-q_0|
\le \frac{r-1}{r}|u-u_0|+\ell^{-A/2}
\le 2\ell^{-A/2}
\le \bigl(\log(d^{1-b}/2)\bigr)^{-A'}.
$
By the induction hypothesis, $J_W$ contains a nonempty subhypergraph $Q$ of average degree at least $q$ such that
\[
\begin{aligned}
|V(Q)|
\le (n-k)\left(\frac{d^{1-b}}2\right)^{-\alpha_{r-1,q}}
       \ell^{O_{r,u_0}(1)}
\le nd^{-b}\ell^{O_{r,u_0}(1)}
 \le nd^{-\alpha_{r,u}}\ell^{O_{r,u_0}(1)}.
\end{aligned}
\]
Replace $W$ by $W\cup V(Q)$ and repeat the procedure with this new set $W$. Each step adds at least one new vertex, and the selected edge sets are pairwise disjoint. 

This process terminates either when $re_J(W)\ge u(|W|+k)$ or when $|W|\ge Bk$. In the former case, we apply Claim~\ref{clm:injective-lifting} and obtain a desired subhypergraph. In the latter case, if the subhypergraphs obtained are $Q_1,\ldots,Q_t$, then
\[
e_J(W)\ge\sum_{i=1}^t e(Q_i)\ge\frac{q}{r-1}\sum_{i=1}^t|V(Q_i)|\ge\frac{q|W|}{r-1},
\]
and hence $re_J(W)/(|W|+k)\ge rqB/((r-1)(B+1))\ge u$. So we can apply Claim~\ref{clm:injective-lifting} and obtain a desired subhypergraph. Note that when this process terminates, $|W|\le Bk+nd^{-\alpha_{r,u}}\ell^{O_{r,u_0}(1)}$. Since $Bk\le nd^{-\alpha_{r,u}}\ell^{O_{r,u_0}(1)}$ and $k\le Bk$, the resulting subhypergraph has order at most $nd^{-\alpha_{r,u}}\ell^{O_{r,u_0}(1)}$.
\end{proof}

Taking $u_0=u=s$ in Proposition~\ref{prop:stable-supercritical} proves
Theorem~\ref{thm:main}.

\begin{remark}
We explain here why Proposition~\ref{prop:stable-supercritical} needs uniform constants over a small interval of target average degrees. In Case~3 of the proof, after compressing an $r$-graph to an $(r-1)$-graph, the target average degree becomes
\[
q=\frac{(r-1)u}{r}+(\log d)^{-A/2}.
\]
The term $\ell^{-A/2}$ in
$q=(r-1)u/r+\ell^{-A/2}$ ensures that the average degree remains at least
$u$ after we add the $k=|X|$ vertices of $X$ in the lifting. 
Since $q$ depends on $d$, a statement proved separately for each fixed target average degree cannot provide uniform constants in the induction.
\end{remark}

\subsection{The high-density regime}
At the critical density $d=n^{1/\alpha_{r,s}}$, Theorem~\ref{thm:lower} shows that the order of the smallest subhypergraph of average degree at least $s$ cannot be bounded independently of $n$. We now show that this is no longer true as soon as the exponent of $n$ is increased by any fixed $\varepsilon>0$: If $d\ge n^{1/\alpha_{r,s}+\varepsilon}$, then such a subhypergraph already exists on $O_{r,s,\varepsilon}(1)$ vertices. Moreover, for fixed $r$, $s$ and $\varepsilon$, such a subhypergraph can be found in time $n^{O_{r,s,\varepsilon}(1)}$.

\begin{proof}[Proof of Theorem~\ref{thm:high-density}]
We argue by induction on $r$. For $r=2$, this is Theorem \ref{thm:JST-high}. Let $r\ge3$ and assume that the theorem holds for $(r-1)$-graphs. Let $H$ be an $n$-vertex $r$-graph with $d:=\avgd(H)\ge n^{1/\alpha_{r,s}+\varepsilon}$. 

First suppose that $r/(r-1)<s<r/(r-2)$. Set $q:=2s/(r-(r-2)s)$. Then $q>2$, $1-2/q=1/\alpha_{r,s}$ and $rq/(2+(r-2)q)=s$. Let $M=\lceil q\rceil$. If some pair has codegree at least $M$, then $M$ edges containing this pair span at most $2+(r-2)M$ vertices and have average degree at least $rM/(2+(r-2)M)\ge s$. Otherwise, choose a pair $p(e)\subseteq e$ for each $e\in E(H)$, and let $G$ be the graph formed by the distinct pairs $p(e)$. Since every pair has codegree at most $M-1$, we have $e(G)\ge e(H)/(M-1)$ and hence $\avgd(G)\ge2d/(r(M-1))\ge n^{1-2/q+\varepsilon/2}$. By Theorem~\ref{thm:JST-high}, $G$ contains a subgraph $F$ of average degree at least $q$ and order at most $T_1=T_1(r,s,\varepsilon)$. Let $t=|V(F)|$ and $m=\lceil qt/2\rceil$, and choose $m$ edges of $F$. For each chosen pair $xy$, select an edge $e_{xy}\in E(H)$ with $p(e_{xy})=xy$, and let $F'$ be the $r$-graph formed by these $m$ edges. These edges are distinct and $|V(F')|\le t+(r-2)m$. Hence $\avgd(F')\ge rm/(t+(r-2)m)\ge rq/(2+(r-2)q)=s$.

Now suppose that $s\ge r/(r-2)$. Write $\beta=1/\alpha_{r,s}=(r-1)-r/s$ and $q_0=(r-1)s/r$. Since $(r-2)-\frac{r-1}{q_0}=(r-2)-\frac rs=\beta-1$, we may choose $q>q_0$ sufficiently close to $q_0$ that $\frac1{\alpha_{r-1,q}}<\beta-1+\frac{\varepsilon}{2}$. In particular, $q>(r-1)/(r-2)$. Set $p=r/n$. For all sufficiently large $n$, we have $0<p\le1/(16r)$ and $pn=r$, so we apply Lemma~\ref{lem:compression} with $\sigma=s$. If Lemma~\ref{lem:compression}~(a) holds, then $H$ contains a subhypergraph of average degree at least $s$ on at most $Cr=O_{r,s}(1)$ vertices, and we are done.

So Lemma~\ref{lem:compression} (b) holds. Then there are a set $X\subseteq V(H)$ with $|X|=r$ and an $(r-1)$-graph $G$ on $V(H)\setminus X$, with $N:=|V(G)|=n-r$ and $\avgd(G)\ge cpd=crd/n\ge crn^{\beta-1+\varepsilon}$, where $c=c(r,s)>0$, such that for every $e\in E(G)$ there is $x_e\in X$ with $e\cup\{x_e\}\in E(H)$. Since $1/\alpha_{r-1,q}<\beta-1+\varepsilon/2$, for all sufficiently large $n$ we have $\avgd(G)\ge N^{1/\alpha_{r-1,q}+\varepsilon/4}$. Applying the induction hypothesis with $r-1$, $q$ and $\varepsilon/4$, we obtain a constant $T'=T'(r,s,\varepsilon)$ such that every $N$-vertex $(r-1)$-graph of average degree at least $N^{1/\alpha_{r-1,q}+\varepsilon/4}$ contains a subhypergraph of average degree at least $q$ on at most $T'$ vertices. Since $q>(r-1)s/r$, we choose an integer $B>s(r-1)/(rq-s(r-1))$. Then $rqB/((r-1)(B+1))>s$.


We fix $G$ and construct a set $W\subseteq V(G)$ iteratively, starting with $W=\varnothing$. If $re_G(W)\ge s(|W|+r)$, then lifting the edges of $G[W]$ through $X$ gives the required subhypergraph. Otherwise, if $|W|<Br$, then $e_G(W)<s(B+1)$. For the current set $W$, define the auxiliary $(r-1)$-graph $G_W:=G-E(G[W])$. Since $\avgd(G_W)=\avgd(G)-(r-1)e_G(W)/N$, we have $\avgd(G_W)\ge N^{1/\alpha_{r-1,q}+\varepsilon/4}$. By the induction hypothesis, $G_W$ contains a nonempty subhypergraph $J$ of average degree at least $q$ and order at most $T'$. Replace $W$ by $W\cup V(J)$ and repeat. Since $G_W$ has no edge contained in the previous $W$, each step adds at least one new vertex, and the edge sets selected in distinct steps are pairwise disjoint.

The process terminates either when $re_G(W)\ge s(|W|+r)$ or when $|W|\ge Br$. In the former case, lifting gives the required subhypergraph. In the latter case, let $J_1,\ldots,J_t$ be the subhypergraphs selected during the iteration. Then
\[
e_G(W)\ge\sum_{i=1}^t e(J_i)\ge\frac{q}{r-1}\sum_{i=1}^t|V(J_i)|\ge\frac{q|W|}{r-1}.
\]
Hence $re_G(W)/(|W|+r)\ge rqB/((r-1)(B+1))>s$, so lifting again gives the required subhypergraph. At the first stopping step, $|W|<Br+T'$, and the lifted subhypergraph has fewer than $(B+1)r+T'$ vertices.

Taking $T$ to be the maximum of the bounds obtained in the two cases proves the existence statement. Finally, the algorithmic statement follows by enumerating all vertex sets $U\subseteq V(H)$ of size at most $T$ and checking whether $\avgd(H[U])\ge s$. Since $T=T(r,s,\varepsilon)$ is independent of $n$, this takes time $n^{O_{r,s,\varepsilon}(1)}$.
\end{proof}

\section*{AI declaration}
ChatGPT 5.6 was used during the development of this work. Its main mathematical contribution concerned Proposition~\ref{prop:stable-supercritical}. In particular, AI-assisted discussions were used to explore the proof strategy and to suggest and check some of the intermediate parameter choices and estimates appearing in its proof. The resulting argument was subsequently reorganized, completed, and independently verified by the authors.

\appendix
\section{Proof of Lemma \ref{lem:JST-uniform}}\label{JST}
\begin{proof}[Proof of Lemma \ref{lem:JST-uniform}]
Let $\eta=\min\{1,(s_0-2)/2\}$ and write $\rho=s/2$.
Since $|s-s_0|\le\eta$, we have $\frac{s_0+2}{4}\le \rho\le \frac{s_0+1}{2}$.
In particular, $\rho$ is bounded above and
$1-1/\rho\ge (s_0-2)/(s_0+2)>0$.

We first check that the constants in the proof of \cite[Theorem~2.13]{JST} can be chosen uniformly in this range. By \cite[Lemma~2.9]{JST} and \cite[Lemma~2.12]{JST}, we take $c_0=1/(16\rho(\lceil2\rho\rceil+3))$, and so $c_0$ is bounded below by a positive constant depending only on $s_0$. Moreover, all the preliminary estimates in the proof of \cite[Theorem~2.13]{JST} can be made uniform in this range. Indeed, $\rho$ is bounded above and $1-1/\rho$ is bounded away from zero, while $\lceil 2\rho\rceil$ is bounded in terms of $s_0$. Thus, for any fixed $C>2$, after choosing $d_0$ sufficiently large in terms of $s_0$, the estimates preceding the final inequality in their proof, including $t\ge20\rho\log d$ and conditions 1--3 of \cite[Lemma~2.12]{JST}, hold uniformly. In the final estimate, the constants $c_1=c_1(\rho)$ and $c_2=c_2(\rho)$ are also bounded above by constants depending only on $s_0$, while $c_0$ is bounded below by a positive constant depending only on $s_0$. Hence it is enough that $(\log d)^{C(1-(1-\varepsilon)/\rho)}$ dominates a constant multiple of $(\log d)^{4-(1-\varepsilon)/\rho}$. Since $1-(1-\varepsilon)/\rho\ge(s_0-2)/(s_0+2)$ and $4-(1-\varepsilon)/\rho<4$, it is enough to choose $C$ such that $C(s_0-2)/(s_0+2)>4$, and then choose $d_0$ sufficiently large. Therefore, uniformly for $|s-s_0|\le\eta$, the proof of \cite[Theorem~2.13]{JST} gives a subgraph $F$ of average degree at least $s$ such that
$
|V(F)|\le N d^{-s/(s-2)}(\log d)^C.
$
\end{proof}

\end{document}